\documentclass[11pt,a4paper]{article}

\usepackage[latin1]{inputenc}
\usepackage{amsmath,amsfonts,amssymb}
\usepackage[margin=2.5cm]{geometry}
\usepackage{graphicx}
\usepackage{hyperref}
\hypersetup{
	colorlinks =true,        % false: boxed links; true: colored links
	linkcolor =blue,         % color of internal links
	citecolor =blue,         % color of links to bibliography
	urlcolor =blue           % color of external links
}

\usepackage{amsthm}
\usepackage{thmtools}
\usepackage[capitalise]{cleveref}
\usepackage{xr}
\usepackage{enumerate}
\usepackage{framed}

\newtheorem{theorem}{Theorem}[section]
\newtheorem{proposition}[theorem]{Proposition}
\newtheorem{corollary}[theorem]{Corollary}
\newtheorem{definition}[theorem]{Definition}
\newtheorem{lemma}{Lemma}[section]
\declaretheorem[style=remark,qed=$\Diamond$,Refname={Example,Examples},sibling=theorem]{example}
\declaretheorem[style=remark,qed=$\Diamond$,Refname={Remark,Remarks},sibling=theorem]{remark}

\allowdisplaybreaks

\newcommand{\setto}{\rightrightarrows} %Set valued map.

\newcommand{\Hilbert}{\mathcal{H}}
\newcommand{\Id}{\operatorname{Id}}
\newcommand{\Fix}{\operatorname{Fix}}

\newcommand{\dist}{\operatorname{d}}

\newcommand{\graph}{\operatorname{gra}}

\usepackage[most]{tcolorbox}
\usepackage[normalem]{ulem}

\begin{document}

\title{Convergence Rates for Boundedly Regular Systems}
\author{Ern\"o Robert Csetnek\thanks{Faculty of Mathematics,
                                      University of Vienna,
                                      Oskar-Morgenstern-Platz 1,
                                      1090 Vienna, \textsc{Austria}.
                                      \mbox{Email:~\href{mailto:ernoe.robert.csetnek@univie.ac.at}
                                           {ernoe.robert.csetnek@univie.ac.at}}}
           \and
         Andrew Eberhard\thanks{Mathematical Sciences,
                                RMIT University,
                                124 La Trobe Street,
                                Melbourne VIC 3000, \textsc{Australia}.
                                Email~\href{mailto:andy.eberhard@rmit.edu.au}
                                           {andy.eberhard@rmit.edu.au}}
           \and
         Matthew K. Tam\thanks{School of Mathematics \& Statistics,
	   	              	       The University of Melbourne,
	   	              	       Parkville VIC 3010, \textsc{Australia}.
	   	              	       Email:~\href{mailto:matthew.tam@unimelb.edu.au}
	   	              	                   {matthew.tam@unimelb.edu.au}}
	   }

\maketitle	

\begin{abstract}
In this work, we consider a continuous dynamical system associated with the fixed point set of a nonexpansive operator which was originally studied by Bo\c{t} \& Csetnek (2015). Our main results establish convergence rates for the system's trajectories when the nonexpansive operator satisfies an additional regularity property. This setting is the natural continuous-time analogue to discrete-time results obtained in Bauschke, Noll \& Phan (2015) and Borwein, Li \& Tam (2017) by using the same regularity properties.
\end{abstract}

\paragraph{Mathematics Subject Classification (MSC2010).} 34G25 $\cdot$ %Evolution inclusions
                     47J25 $\cdot$ %Iterative procedures involving nonlinear operators
                     90C25         %Convex programming

\section{Introduction}
Let $\Hilbert$ denote a real Hilbert space with inner-product $\langle\cdot,\cdot\rangle$ and induced norm $\|\cdot\|$. In this work, we consider the continuous-time dynamical system with initial point $x_0\in\Hilbert$ given by
  \begin{equation}\label{eq:system}
   \dot{x}(t) = \lambda(t)\left(T(x(t))-x(t)\right),\quad x(0)=x_0,
  \end{equation} 
where $T\colon\Hilbert\to\Hilbert$ is nonexpansive and $\lambda\colon[0,+\infty)\to[0,1]$ is Lebesgue measurable. We remark that the parameter function $\lambda$ has an interpretation as a time-scaling factor, through which \eqref{eq:system} can be shown equivalent to the case with $\lambda(t)=1$ for all $t\geq 0$. For details, see \cite[Section~4]{BS2015}.

We shall investigate the behaviour of trajectories of \eqref{eq:system} which are understood in the sense of \emph{strong global solutions}.
\begin{definition}[Strong global solution]\label{d:strong global soln}
A trajectory $x\colon[0,+\infty)\to\Hilbert$ is a \emph{strong global solution} of \eqref{eq:system} if the following properties are satisfied:
\begin{enumerate}[(i)]
 \item $x$ is absolutely continuous on each interval $[0,b]$ for $0<b<+\infty$.
 \item $\dot{x}(t)=\lambda(t)\bigl( T(x(t))-x(t) \bigr)$ for almost all $t\in[0,+\infty)$.
 \item $x(0)=x_0$.
\end{enumerate}
\end{definition}  

Here, absolute continuity of the trajectory $x$ on $[0,b]$ is understood in the vector-valued sense (see, for instance, \cite[Definition~2.1]{AS}) which implies 
  $$ x(t) = x(0) + \int_0^t\dot{x}(s)\,ds\quad\forall t\in[0,b]. $$  
The existence and uniqueness of a strong global solution for each $x_0\in\Hilbert$ follows as a consequence of the \emph{Cauchy--Lipschitz theorem}. The detailed argument can be found in \cite[Section~2]{BS2015}.

Convergence of these trajectories (without rates) was established by Bo\c{t} \& Csetnek \cite{BS2015}.  
\begin{theorem}[{\cite[Theorem~6]{BS2015}}]\label{th:botcsetnek}
Suppose $T\colon\Hilbert\to\Hilbert$ is nonexpansive with $\Fix T\neq\emptyset$ and $\lambda\colon[0,+\infty)\to[0,1]$ be Lebesgue measurable with either
 $$ \int_0^{+\infty}\lambda(t)\bigl( 1-\lambda(t) \bigr)\,dt = +\infty\text{~~or~~}\inf_{t\geq 0}\lambda(t)>0. $$
 Let $x$ denote the unique strong global solution of \eqref{eq:system}. Then the following assertions hold.
  \begin{enumerate}[(i)]
    \item The trajectory $x$ is bounded and $\int_0^{+\infty}\|\dot{x}(t)\|^2dt<+\infty$.
    \item $\lim_{t\to+\infty}\left( T(x(t))-x(t)\right) =0. $
    \item $\lim_{t\to+\infty}\dot{x}(t)= 0$.
    \item $x(t)$ converges weakly to a point $\bar{x}\in\Fix T$ as $t\to+\infty$.
  \end{enumerate}
\end{theorem}

The dynamical system \eqref{eq:system} can be viewed as a continuous-time analogue to the discrete-time system given by
  \begin{equation}\label{eq:discrete system}
   x_{k+1} = (1-\lambda_k)x_k+\lambda_k T(x_k).
  \end{equation}    
More precisely, the sequence $(x_k)$ in \eqref{eq:discrete system} can be viewed as a discretisation of the trajectory $x(t)$ in \eqref{eq:system} along unit stepsizes. In other words, for $k\in\mathbb{N}$, we take $\lambda_k\approx\lambda(k)$ and $x_k\approx x(k)$ together with the forward discretisation $\dot{x}(k)\approx x_{k+1}-x_k$. In the literature, the discrete system \eqref{eq:discrete system} is well-known as the \emph{Krasnoselskii--Mann} iteration \cite{BCP2019} corresponding to $T$. By choosing the operator $T$ appropriately, many iterative algorithms can be understood within this framework (see, for instance, \cite[Section~26]{BauCom2nd}).

In analogue with Theorem~\ref{th:botcsetnek},  it can be shown that the sequence $(x_k)_{k\in\mathbb{N}}$ generated by \eqref{eq:discrete system} converges weakly to a point in $\Fix T$ provided that $(\lambda_k)$ satisfies $\sum_{k=1}^\infty\lambda_k(1-\lambda_k)=+\infty$ \cite[Theorem~5.15]{BauCom2nd}. Furthermore, when $T$ satisfies appropriate regularity conditions, information about the rate of convergence of $(x_k)$ can also be provided -- it converges $R$-linearly when $T$ is \emph{boundedly linearly regular}, and sublinearly when $T$ is \emph{boundedly H\"older regular}. Although we defer formally defining these regularity notions until Section~\ref{s:boundedly regular ops}, we will nevertheless state the following result for completeness.

\begin{theorem}\label{th:discrete}
Let $T\colon\Hilbert\to\Hilbert$ be an nonexpansive operator with $\Fix T\neq\emptyset$. 
Let $x_0\in\Hilbert$ and consider the sequence $(x_k)$ given by \eqref{eq:discrete system} with $(\lambda_k)\subseteq [0,1]$ such that $\inf_{k\in\mathbb{N}}\lambda_k(1-\lambda_k)>0$. Then there exists a point $\overline{x}\in\Fix T$ such that the following assertions hold.
\begin{enumerate}[(i)]
\item If $T$ is boundedly linearly regular, then $x_k\to\bar{x}$ with at least $R$-linear rate, that is, with at least rate $O(r^k)$ for some $r\in[0,1)$.
\item If $T$ is boundedly H\"older regular, then $x_k\to\bar{x}$ with at least rate $O(k^{-\rho})$ for some $\rho>0$.
\end{enumerate}
\end{theorem}
\begin{proof}
(i):~See \cite[Theorem~6.1]{BNP2015}. (ii):~See \cite[Corollary~3.9]{BLT2017}. For generalisations, see \cite{LNT}.
\end{proof}

In this work, we show that the analogues statements about convergence rates given in Theorem~\ref{th:discrete} also hold in the continuous-time setting. From the perspective of iterative algorithms in optimisation, understanding the interplay between the corresponding discrete and continuous-time systems provides insight into the conditions required convergence as well as a technology for derive new schemes. For specific examples, see \cite{CMT,RT}. For other recent works which study the interplay between discrete and continuous-time systems, the reader is referred to \cite{RYY,PS,abbas2014newton,attouch2018convergence,banert2018forward}.

\bigskip

The remainder of this work is structured as follows. In Section~\ref{s:boundedly regular ops}, we review notions of bounded regularity for operators. These notions are then used in Section~\ref{s:convergence} to prove convergence rates for the strong global trajectories of \eqref{eq:system}. Closure properties of the classes of boundedly regular operators are studied in Section~\ref{s:closure properties}. These properties are of interest in their own right and complement the results in \cite{CRZ2018}. Finally, Section~\ref{s:convergence2} uses these closure properties to deduce several extensions of the results from Section~\ref{s:convergence}.

\section{Boundedly Regular Operators}\label{s:boundedly regular ops}
In this section, we recall two notions of boundedly regular operators as well as providing examples of each. These notions are a kind of \emph{error bound} in that, when satisfied, that they bound the distance to the fixed point set of an operator in terms of its residual.

 The first notion, based on linear regularity, was proposed for projection operators by Bauschke \& Borwein \cite{BB96} and for the general case by Bauschke, Noll \& Phan \cite{BNP2015}. 
\begin{definition}[Linearly regular operators]\label{d:lin reg ops}
An operator $T\colon\Hilbert\to\Hilbert$ is \emph{linearly regular} on $U\subseteq\Hilbert$ if there exists a constant $\kappa>0$ such that
  $$ \dist(y,\Fix T)\leq\kappa \|y-T(y)\|\quad\forall y\in U. $$
If $T$ is linearly regular on every bounded subset of $\Hilbert$, it is said to be \emph{boundedly linearly regular}.
\end{definition}

Recall that a set is \emph{polyhedral} if it can be expressed as the intersection of finitely many closed half-spaces and/or hyperplanes, and that an operator is \emph{polyhedral} if its graph is the union of finitely many polyhedral sets. For remarks on this terminology, see \cite[p.~76]{RW2009}.
\begin{proposition}\label{prop:poly}
Let $\Hilbert=\mathbb{R}^n$. If $T\colon\Hilbert\to\Hilbert$ is polyhedral with $\Fix T\neq\emptyset$, then $T$ is boundedly linearly regular.
\end{proposition} 
\begin{proof}
Since $\Id$ and $T$ are polyhedral and the class of polyhedral operators is closed under addition \cite[p.~206]{R1981}, the operator $F:=\Id-T$ is also polyhedral. By \cite[Corollary]{R1981} applied to $F$, there exist $\kappa_1>0$ and $\epsilon>0$ such that 
\begin{equation}\label{eq:prop poly1}
 \dist(x,\Fix T)=\dist(x,F^{-1}(0)) \leq \kappa_1 \dist(0,F(x)) = \kappa_1 \|x-T(x)\|
\end{equation}
for all $x\in\Hilbert$ with $\|x-T(x)\|<\epsilon$.
Let $U\subseteq\Hilbert$ be a nonempty bounded set. Then $\kappa_2:=\sup_{x\in U}\dist(x,\Fix T) < +\infty$. Thus, for all $x\in U$ with $\|x-T(x)\|\geq \epsilon$, we have
\begin{equation}\label{eq:prop poly2}
  \frac{\dist(x,\Fix T)}{\|x-T(x)\|} \leq \frac{\dist(x,\Fix T)}{\epsilon} \leq \frac{\kappa_2}{\epsilon}. 
\end{equation}
By combining \eqref{eq:prop poly1} and \eqref{eq:prop poly2}, we deduce
$$ \dist(x,\Fix T) \leq \max\left\{\kappa_1,\frac{\kappa_2}{\epsilon}\right\} \|x-T(x)\|\quad \forall x\in U,$$
which establishes the claimed result.
\end{proof}

One drawback of linear regularity is that is often too restrictive to hold or too difficult to verify in practice (\emph{i.e.,}~beyond polyhedral settings such as Example~\ref{prop:poly}). For further examples, see \cite[Section~2]{BNP2015}. To overcome this shortcoming, the following H\"older counterpart of Definition~\ref{d:lin reg ops} was introduced in \cite[Definition~2.7]{BLT2017}.
\begin{definition}[H\"older regular operators]
An operator $T\colon\Hilbert\to\Hilbert$ is \emph{H\"older regular} on $U\subseteq\Hilbert$ if there exists a constants $\kappa>0$ and $\gamma\in(0,1)$ such that
  $$ \dist(y,\Fix T)\leq\kappa \|y-T(y)\|^\gamma \quad\forall y\in U. $$
If $T$ is H\"older regular on every bounded subset of $\Hilbert$, it is said to be \emph{boundedly H\"older regular}. 
\end{definition}

Recall that a set is \emph{semi-algebraic} if it can be expressed as the union of finitely many sets, each of which can be defined by finitely many polynomial equalities and inequalities. An operator is \emph{semi-algebraic} if its graph is a semi-algebraic set.
\begin{proposition}\label{prop:semi-alg op}
Let $\Hilbert=\mathbb{R}^n$. If $T\colon\Hilbert\to\Hilbert$ is continuous and semi-algebraic with $\Fix T\neq\emptyset$, then $T$ is boundedly H\"older regular.
\end{proposition}
\begin{proof}
Let $U$ be a nonempty bounded set. Then there exists an $R>0$ such that 
 $$U\subseteq \mathbb{B}(0,R):=\{x\in\Hilbert:\|x\|\leq R\}$$
where we note that $\mathbb{B}(0,R)$ is semi-algebraic. Consider the continuous functions 
  $$\phi(y):=\|y-T(y)\|\text{~~and~~}\psi(y):=\dist(y,\Fix T).$$
Since $\|\cdot\|$ and $\Id-T$ are semi-algebraic as, their composition, the function $\phi$ is also semi-algebraic \cite[Proposition~2.2.6(i)]{BCF}. Since $\Fix T=\phi^{-1}(0)$ and $\phi$ is semi-algebraic, the set $\Fix T$ is also semi-algebraic by \cite[Proposition~2.2.7]{BCF}. By \cite[Proposition~2.2.8(i)]{BCF}, it then follows that $\dist(\cdot,\Fix T)$ is semi-algebraic. Thus, since $\phi$ and $\psi$ are continuous semi-algebraic functions with $\psi^{-1}(0)=\phi^{-1}(0)=\Fix T\neq\emptyset$, {\L}ojasiewicz's inequality \cite[Corollary~2.6.7]{BCF} implies that there exists constants $\kappa>0$ and $\gamma\in(0,1)$ such that 
  $$ \dist(x,\Fix T)=|\psi(x)| \leq \kappa|\phi(x)|^\gamma =\kappa\|x-T(x)\|^\gamma \quad\forall x\in\mathbb{B}(0,R), $$
and the claimed result follows.
\end{proof}

\begin{example}[Forward-backward operator]
 Let $\Hilbert=\mathbb{R}^n$ and consider the \emph{monotone inclusion}
   \begin{equation}\label{eq:monotone inclusion}
    0\in(A+B)(x), 
   \end{equation}
where $A\colon\Hilbert\setto\Hilbert$ is maximally monotone and $B\colon\Hilbert\to\Hilbert$ is monotone and continuous. This problem arises, for instance, as the optimality conditions of the minimisation problem 
  \begin{equation}\label{eq:min}
   \min_{x\in\Hilbert}g(x)+f(x),
  \end{equation}
where $g\colon\Hilbert\to(-\infty,+\infty]$ is proper, lsc, convex and $f\colon\Hilbert\to\mathbb{R}$ is convex and differentiable. More precisely, by setting $A=\partial g$ (\emph{i.e.,} the convex subdifferential of $g$) and $B=\nabla f$.

The \emph{forward-backward operator} $T\colon\Hilbert\to\Hilbert$ for \eqref{eq:monotone inclusion} with stepsize $\lambda>0$ is given by 
 $$ T := (\Id+\lambda A)^{-1}\circ(\Id-\lambda B), $$
where the \emph{resolvent operator} $(\Id+\lambda A)^{-1}$ is single-valued and continuous with full domain \cite[Proposition~23.10]{BauCom2nd}. Then $T$ is continuous and $\Fix T=(A+B)^{-1}(0)$. Moreover, $T$ is semi-algebraic, and hence boundedly H\"older regular by Proposition~\ref{prop:semi-alg op}, whenever $A$ and $B$ are semi-algebraic. Indeed, if $A$ and $B$ are semi-algebraic, then so are $\Id+\lambda A$ and $\Id-\lambda B$. And, since $(u,v)\in\graph(\Id+\lambda A)$ if and only if $(v,u)\in\graph(\Id+\lambda A)^{-1}$, the resolvent operator is also semi-algebraic. As the composition of two semi-algebraic operators, $T$ is therefore also semi-algebraic.

Since the subdifferential of a convex semi-algebraic function is again semi-algebraic (see, for instance, \cite{I,DL}), we also note that, in particular, the forward-backward operator applied to \eqref{eq:min} is boundedly H\"older regular when $f$ and $g$ are semi-algebraic.
\end{example}

\begin{remark}\label{re:bounded reg equiv}
Let $T\colon\Hilbert\to\Hilbert$ and let $z\in\Hilbert$. Then, it is immediate from the respective definitions, that  $T$ is boundedly linearly (resp.\ H\"older) regular if and only if $T$ is boundedly linearly (resp.\ H\"older) regular on $\mathbb{B}(z,R)$ for all $R>0$.
\end{remark}

\section{Convergence of trajectories with regularity}\label{s:convergence}
In this section, we show a refinement of Theorem~\ref{th:botcsetnek}. Namely, that the convergence rate of the trajectories in \eqref{eq:system} can be given when the operator $T$ is boundedly regular. Although it will not always be explicitly stated within this section's proofs to avoid repetition, identities will sometime be understood to hold for almost all $t\in[0,+\infty)$ due the identity for $\dot{x}$ in Definition~\ref{d:strong global soln}(ii).

We shall require the following lemmata as well as the well-known identity:
\begin{equation}\label{eq:affine comb}
 \|(1-\alpha)u+\alpha v\|^2 + \alpha(1-\alpha)\|u-v\|^2 = (1-\alpha)\|u\|^2+\alpha\|v\|^2\quad \forall \alpha\in\mathbb{R},\,\forall u,v\in\Hilbert.
\end{equation}

\begin{lemma}\label{l:avg}
Let $x$ be the unique strong global solution of \eqref{eq:system}, let $x^*\in\Fix T$ and 
suppose $\inf_{t\geq 0}\lambda(t)>0$. For almost all $t\in[0,+\infty)$, we have
$$ \|\dot{x}(t)+x(t)-x^*\|^2 + \frac{1-\lambda(t)}{\lambda(t)}\|\dot{x}(t)\|^2 \leq \|x(t)-x^*\|^2. $$
\end{lemma}
\begin{proof}
By applying \eqref{eq:affine comb} followed by nonexpansivity of $T$, we obtain
\begin{align*}
 &\|\dot{x}(t)+x(t)-x^*\|^2 \\
  &\quad= \|(1-\lambda(t))(x(t)-x^*)+\lambda(t)(T(x(t))-x^*)\|^2 \\
  &\quad= (1-\lambda(t))\|x(t)-x^*\|^2 + \lambda(t)\|T(x(t))-x^*\|^2 - \lambda(t)(1-\lambda(t))\|x(t)-T(x(t))\|^2 \\
  &\quad\leq \|x(t)-x^*\|^2 - \frac{1-\lambda(t)}{\lambda(t)}\|\dot{x}(t)\|^2,
\end{align*} 
which completes the proof of the result.
\end{proof}

\begin{proposition}[{\cite[Corollary~12.31]{BauCom2nd}}]\label{prop:grad proj}
Let $C\subseteq\Hilbert$ be a nonempty closed convex set. Then $x\mapsto\dist^2(x,C)$ is Fr\'echet differentiable on $\Hilbert$ with $\nabla\dist^2(\cdot,C)=2(\Id-P_C)$.
\end{proposition}

\begin{lemma}\label{l:ineqs}
Let $x$ be the unique strong global solution of \eqref{eq:system}. Suppose $\Fix T\neq\emptyset$ and $\inf_{t\geq 0}\lambda(t)>0$. Then, for almost all $t\in[0,+\infty)$, we have
\begin{enumerate}[(i)]
\item\label{l:ineqs:FixT} $\displaystyle\frac{d}{dt}\dist^2(x(t),\Fix T)  \leq -\lambda(t)\|x(t)-T(x(t))\|^2$, and
\item\label{l:ineqs:xbar}  $\displaystyle\frac{d}{dt}\|x(t)-x^*\|^2 
   \leq  -\lambda(t)(1-\lambda(t))\|x(t)-T(x(t))\|^2- \|\dot{x}(t)\|^2$ for all $x^*\in\Fix T$.
\end{enumerate}   
\end{lemma}
\begin{proof}
\eqref{l:ineqs:FixT}:~Since $T$ is nonexpansive, $F:=\Fix T$ is nonempty, closed and convex \cite[Proposition~4.13]{BauCom2nd}. The chain-rule together with Proposition~\ref{prop:grad proj} therefore implies
\begin{equation*}
\begin{aligned}
\frac{d}{dt}\dist^2(x(t),F) 
&= \langle \dot{x}(t),\nabla \dist^2(\cdot,F)(x(t)) \rangle \\
&= 2\langle \dot{x}(t),x(t)-P_F(x(t))\rangle \\
&= \|\dot{x}(t)+x(t)-P_F(x(t))\|^2-\|\dot{x}(t)\|^2-\|x(t)-P_F(x(t))\|^2 \\
&= \|\dot{x}(t)+x(t)-P_F(x(t))\|^2-\lambda(t)^2\|x(t)-T(x(t))\|^2-\|x(t)-P_F(x(t))\|^2.
\end{aligned}
\end{equation*}
Since $P_F(x(t))\in F=\Fix T$, Lemma~\ref{l:avg} then gives
\begin{equation*} %\label{eq:avg}
\|\dot{x}(t)+x(t)-P_F(x(t))\|^2 \leq \|x(t)-P_F(x(t))\|^2 - \lambda(t)\bigl(1-\lambda(t)\bigr)\|x(t)-T(x(t))\|^2.
\end{equation*}
The claimed inequality follows by combining the previous two equations.

\eqref{l:ineqs:xbar}:~For any $\bar{x}\in\Fix T$, we have
\begin{align*}
 \frac{d}{dt}\|x(t)-\bar{x}\|^2 
   = 2\langle\dot{x}(t),x(t)-\bar{x}\rangle \
   &= \|\dot{x}(t)+x(t)-\bar{x}\|^2 - \|\dot{x}(t)\|^2 - \|x(t)-\bar{x}\|^2.
\end{align*}
The result then follows by combining this equality with Lemma~\ref{l:avg}.
\end{proof}

We shall also require the following well-known, classical result.
\begin{lemma}[Gr\"onwall's inequality]\label{l:groenwall}
Let $u\colon[0,+\infty)\to[0,+\infty)$ be absolutely continuous. Suppose there exists $\alpha>0$ such that, for almost all $t\in[0,+\infty)$, we have
 $$ \frac{d}{dt}u(t) \leq -\alpha u(t). $$
Then $u(t) \leq \exp(-\alpha t)u(0)$ for all $t\in[0,+\infty)$.
\end{lemma}

The following theorem is our first main result. It shows that the dynamical system \eqref{eq:system} is \emph{exponentially stable} when $T$ is boundedly linearly regular.

\begin{theorem}\label{th:linear reg convergence}
Suppose $T\colon\Hilbert\to\Hilbert$ is nonexpansive with $\Fix T\neq\emptyset$ and $\lambda\colon[0,+\infty)\to[0,1]$ is Lebesgue measurable with $\lambda^\ast:=\inf_{t\geq 0}\lambda(t)>0$. Let $x$ be the unique strong global solution of \eqref{eq:system}. If $T$ is boundedly linearly regular, then there exists $\bar{x}\in\Fix T$ and $\kappa>0$ such that, for almost all $t\in[0,+\infty)$, we have
  $$ \|x(t)-\bar{x}\| \leq 2\exp\left(-\frac{\lambda^*}{2\kappa^2}t\right)\dist(x_0,\Fix T). $$ 
That is, the trajectory $x(t)$ converges exponentially to $\bar{x}$ as $t\to+\infty$.
\end{theorem}
\begin{proof}
By Theorem~\ref{th:botcsetnek}, the trajectory $x$ is bounded and $x(t)$ converges weakly to a point $\bar{x}\in\Fix T$ as $t\to+\infty$. Thus, since $T$ is boundedly linearly regular, there exists $\kappa>0$ such that
$$ \dist(x(t),\Fix T)\leq \kappa \|x(t)-T\big(x(t)\bigr)\|. $$
Combining this with Lemma~\ref{l:ineqs}\eqref{l:ineqs:FixT} yields
  $$ \frac{d}{dt}\dist^2(x(t),\Fix T)  \leq -\lambda(t)\|x(t)-T(x(t))\|^2 \leq -\frac{\lambda^*}{\kappa^2}\dist^2(x(t),\Fix T). $$  
By applying Gr\"onwall's inequality (Lemma~\ref{l:groenwall}) to the function $t\mapsto\dist^2(x(t),\Fix T)$, we obtain
  \begin{equation}
  \label{eq:key1}\dist^2(x(t),\Fix T)\leq \exp\left(-\frac{\lambda^*}{\kappa^2}t\right)\dist^2(x_0,\Fix T). 
  \end{equation} 
Let $x^*\in\Fix T$ be arbitrary. By Lemma~\ref{l:ineqs}\eqref{l:ineqs:xbar}, we have $\frac{d}{dt}\|x(t)-x^*\|^2\leq 0$ and hence the function $t\mapsto\|x(t)-x^*\|^2$ is nonincreasing. Assuming that $s>t$, we deduce
$$ \|x(t)-x(s)\| \leq \|x(t)-x^*\| + \|x(s)-x^*\|  \leq 2\|x(t)-x^*\|. $$
Using weak lower semicontinuity of the norm and setting $x^*=P_{\Fix T}\bigl(x(t)\bigr)$ then gives
\begin{equation}\label{eq:key2}
\|x(t)-\bar{x}\| \leq \liminf_{s\to+\infty}\|x(t)-x(s)\| \leq 2\dist(x(t),\Fix T).
\end{equation}
The result then follows by combining \eqref{eq:key1} and \eqref{eq:key2}.
\end{proof}

In the following theorem, we make use of the following generalisation of Gr\"onwall's inequality.
\begin{lemma}[Bihari--LaSalle inequality]\label{l:bls}
Let $u\colon[0,+\infty)\to[0,+\infty)$ be absolutely continuous. 
Suppose there exists $\alpha>0$ and $\gamma\in(0,1)$ such that, for almost all $t\in[0,+\infty)$, we have
 \begin{equation}\label{eq:bls}
 \frac{d}{dt}u(t)\leq -\alpha u(t)^{\frac{1}{\gamma}}.
 \end{equation}
Then there exists a constant $M>0$ such that 
  $ u(t) \leq Mt^{-\frac{\gamma}{1-\gamma}} $ for all $t\in[0,+\infty)$.
\end{lemma}
\begin{proof}
If there exists $t_0\geq 0$ such that $u(t_0)=0$, then \eqref{eq:bls} implies that $u(t)=0$ for all $t\geq t_0$ and the result trivially holds. Thus, we suppose that $u>0$. In this case, since $1-1/\gamma<0$, we have
\begin{align*}
\frac{d}{dt}\left( u(t)^{1-\frac{1}{\gamma}} + \left(1-\frac{1}{\gamma}\right)\alpha t\right) 
&= \left(1-\frac{1}{\gamma}\right)u(t)^{-\frac{1}{\gamma}}\left( \frac{d}{dt}u(t) + \alpha u(t)^\frac{1}{\gamma} \right) \geq 0. 
\end{align*}
Thus, since $t\mapsto u(t)^{1-\frac{1}{\gamma}} + \left(1-\frac{1}{\gamma}\right)\alpha t$ is non-decreasing and absolutely continuous, we have
 $$ u(t)^{1-\frac{1}{\gamma}} + \left(1-\frac{1}{\gamma}\right)\alpha t \geq u(0)^{1-\frac{1}{\gamma}} \geq 0, $$
which implies 
  $$ u(t) \leq \left(\frac{\gamma}{\alpha(1-\gamma)}\right)^\frac{\gamma}{1-\gamma}t^{-\frac{\gamma}{1-\gamma}}. $$
This establishes the result and completes the proof.  
\end{proof}

The following theorem is the H\"older regular analogue of Theorem~\ref{th:linear reg convergence}. It is our second main result.

\begin{theorem}\label{th:hoelder reg convergence}
Suppose $T\colon\Hilbert\to\Hilbert$ is nonexpansive with $\Fix T\neq\emptyset$ and $\lambda\colon[0,+\infty)\to[0,1]$ is Lebesgue measurable with $\lambda^\ast:=\inf_{t\geq 0}\lambda(t)>0$. Let $x$ be the unique strong global solution of \eqref{eq:system}. If $T$ is boundedly H\"older regular, then there exists $\bar{x}\in\Fix T$, $M>0$ and $\gamma\in(0,1)$ such that, for almost all $t\in[0,+\infty)$, we have
$$ \|x(t)-\bar{x}\| \leq M\,t^{-\frac{\gamma}{2(1-\gamma)}}. $$
That is, the trajectory $x(t)$ converges with order $\rho:=\frac{\gamma}{2(1-\gamma)}>0$ to $\bar{x}$ as $t\to+\infty$.
\end{theorem}
\begin{proof}
By Theorem~\ref{th:botcsetnek}, $x(t)$ converges weakly to a point $\bar{x}\in\Fix T$. In particular, the trajectory $x$ is bounded and hence, as $T$ is boundedly H\"older regular, there exists $\kappa>0$ and $\gamma\in(0,1)$ such that
$$ \dist(x(t),\Fix T)\leq \kappa \|x(t)-T\bigl(x(t)\bigr)\|^\gamma. $$
Combining this with Lemma~\ref{l:ineqs}\eqref{l:ineqs:FixT} yields
  $$ \frac{d}{dt}\dist^2(x(t),\Fix T)  \leq -\lambda(t)\|x(t)-T(x(t))\|^2 \leq -\frac{\lambda^*}{\kappa^{2/\gamma}}\dist^{2/\gamma}(x(t),\Fix T). $$  
By applying the Bihari--LaSalle inequality (Lemma~\ref{l:bls}) to the function $t\mapsto\dist^2(x(t),\Fix T)$, we deduce the existence of a constant $M_0>0$ such that
  \begin{equation*}
  \dist(x(t),\Fix T)\leq M_0\,t^{-\frac{\gamma}{2(1-\gamma)}}. 
  \end{equation*} 
Let $x^*\in\Fix T$ be arbitrary. By using the same argument as used in Theorem~\ref{th:linear reg convergence} to obtain \eqref{eq:key2}, we deduce
  $$ \|x(t)-\bar{x}\| \leq 2\dist(x(t),\Fix T). $$
Combining the previous two inequalities gives
  $$ \|x(t)-\bar{x}\| \leq M\,t^{-\frac{\gamma}{2(1-\gamma)}}\text{~~where~~}M:=2M_0, $$
which completes the proof.  
\end{proof}

An interesting direction for further investigation would be to study convergence rates under regularity properties for $T$ for second order dynamical systems. Specially, given initial points $u_0,v_0\in\Hilbert$, it is natural to consider the system
\begin{equation}\label{dyn-syst-sec-order-nonexp}\left\{
\begin{array}{ll}
\ddot x(t) + \gamma\dot x(t) + \lambda(t)\big(x(t)-T(x(t))\big)=0\\
x(0)=u_0,\quad \dot x(0)=v_0
\end{array}\right.\end{equation}
where $\gamma>0$ and $\lambda>0$ is as considered above.

The motivation for studying \eqref{dyn-syst-sec-order-nonexp} is that its time discretisation leads to iterative schemes involving
inertial effects, which have had a great impact in the research community
due to the works of Polyak \cite{polyak}, Nesterov \cite{Nes83,nes}, etc. In the particular case when $\lambda(t)=1$ for all $t \in [0, +\infty)$, the dynamical system \eqref{dyn-syst-sec-order-nonexp} has been
investigated in \cite[Theorem 3.2]{att-alv} (see also \cite{b-c-sicon2016}).

\section{Further Properties of Regular Operators}\label{s:closure properties}
In this section, we study closure properties of the classes of boundedly linearly/H\"older regular operators under convex combinations and compositions. In order to establish these properties, we shall work with the following class which includes averaged nonexpansive operators as a special cases.
\begin{definition}[Strongly quasinonexpansive operators]
An operator $T\colon\Hilbert\to\Hilbert$ is \emph{$\rho$-strongly quasinonexpansive ($\rho$-SQNE)} if $\rho>0$ and 
  $$ \|T(x)-x^*\|^2 + \rho\|x-T(x)\|^2 \leq \|x-x^*\|^2 \quad\forall x\in\Hilbert,\,\forall x^*\in \Fix T. $$
\end{definition}
\begin{remark}
Although this paper will only use the results from this section applied to averaged-nonexpansive operators, our motivation for studying SQNE operators is two-fold. Firstly, as every averaged-nonexpansive operator is also SNQE, there is no loss of generality in considering SNQE operators. Secondly, many of the results in this section can be seen as extensions of those in \cite{CRZ2018} which considered SNQE operators. Thus it is natural for us to consider the same setting.  
\end{remark}

The fixed points of SQNE operators satisfy the following properties.
\begin{proposition}[{\cite[Theorem~2.1.26]{Ceg}}]
Let $T_i\colon\Hilbert\to\Hilbert$ be SQNE with a common fixed point. Then the identity
 $$ \Fix T=\cap_{i=1}^n\Fix T_i $$
holds provided that $T$ has one of the following forms:
\begin{enumerate}[(i)]
\item $T=\sum_{i=1}^n\omega_iT_i$ with $\sum_{i=1}^n\omega_i=1$ and $\omega_i>0$ for all $i\in\{1,\dots,n\}$.
\item $T=T_n\dots T_2T_1$
\end{enumerate}
\end{proposition}

We also require the following regularity notion for collections of sets. 
\begin{definition}[Linearly regular collections of sets]\label{d:lin reg inter}
A collection of sets $\{C_1,\dots,C_n\}$ is \emph{linearly regular} on $U$ if there exists $\tau>0$ such that
$$ \dist(x,\cap_{i=1}^n C_i)\leq \tau\max_{i=1,\dots,n}\dist(x, C_i)\quad \forall x\in U. $$
If the collection $\{C_1,\dots,C_n\}$ is linearly regular on every bounded subset of $\Hilbert$, it is said to be \emph{boundedly linearly regular}.
\end{definition}

\begin{theorem}[{\cite[Corollary~5.3]{CRZ2018}}]\label{th:linear reg comb}
Let $T_i\colon\Hilbert\to\Hilbert$ be $\rho_i$-SQNE for all $i\in\{1,\dots,n\}$. Assume $\min_{i=1,\dots,n}\omega_i>0$, $\sum_{i=1}^n\omega_i=1$ and $\cap_{i=1}^n\Fix T_i\neq\emptyset$. Suppose the following assertions hold.
\begin{enumerate}[(i)]
\item The operator $T_i$ is boundedly linearly regular for all $i\in\{1,\dots,n\}$.
\item The collection $\{\Fix T_i\}_{i=1}^n$ is boundedly linearly regular. 
\end{enumerate}
Then $T:=\sum_{i=1}^n\omega_i T_i$ is boundedly linearly regular.
\end{theorem}

\begin{theorem}[{\cite[Corollary~5.6]{CRZ2018}}]\label{th:linear reg comp}
Let $T_i\colon\Hilbert\to\Hilbert$ be $\rho_i$-SQNE for all $i\in\{1,\dots,n\}$. Assume $\cap_{i=1}^n\Fix T_i\neq\emptyset$ and denote $U:=\mathbb{B}(z,R)$ where $z\in\cap_{i=1}^n\Fix T_i$ and $R>0$. 
Suppose the following assertions hold.
\begin{enumerate}[(i)]
\item The operator $T_i$ is linearly regular on $U$ for all $i\in\{1,\dots,n\}$.
\item The collection $\{\Fix T_i\}_{i=1}^n$ is linearly regular on $U$.
\end{enumerate}
Then $T:=T_n\dots T_2T_1$ is linearly regular on $U$.
\end{theorem}

The following is immediate from the definitions.
\begin{corollary}\label{cor:linear reg comp}
Let $T_i\colon\Hilbert\to\Hilbert$ be $\rho_i$-SQNE for all $i\in\{1,\dots,n\}$. Assume $\cap_{i=1}^n\Fix T_i\neq\emptyset$. Suppose the following assertions hold.
\begin{enumerate}[(i)]
\item The operator $T_i$ is boundedly linearly regular for all $i\in\{1,\dots,n\}$.
\item The collection $\{\Fix T_i\}_{i=1}^n$ is boundedly linearly regular.
\end{enumerate}
Then $T:=T_n\dots T_2T_1$ is boundedly linearly regular.
\end{corollary}
\begin{proof}
Follows by combining Theorem~\ref{th:linear reg comp} with Remark~\ref{re:bounded reg equiv}.
\end{proof}

The following regularity notion is the H\"older analogue of Definition~\ref{d:lin reg inter}.

\begin{definition}[H\"older regular collections of sets]
A collection of sets $\{C_1,\dots,C_n\}$ is \emph{H\"older regular} on $U$ if there exists $\tau>0$ and $\theta\in(0,1)$ such that
$$ \dist(x,\cap_{i=1}^n C_i)\leq \tau\max_{i=1,\dots,n}\dist(x, C_i)^\theta\quad \forall x\in U. $$
If the collection $\{C_1,\dots,C_n\}$ is H\"older regular on every bounded subset of $\Hilbert$, it is said to be \emph{boundedly H\"older regular}.
\end{definition}

\begin{lemma}\label{l:hoelder trick}
Let $0<\gamma\leq\theta$ and $b>0$. There exists $M>0$ such that
 $\alpha^\theta \leq M \alpha^\gamma$ for all $\alpha\in[0,b]$.
\end{lemma}
\begin{proof}
Since $\theta-\gamma>0$ by assumption, we have $\alpha^{\theta-\gamma}\leq b^{\theta-\gamma}$. 
Thus, for all $\alpha\in[0,b]$,  we have $\alpha^\theta = \alpha^{\theta-\gamma}\alpha^\gamma \leq M\alpha^\gamma$ with $M=b^{\theta-\gamma}$.
\end{proof}

The following lemma is due to Cegielski \& Zalas \cite{CZ}. Since we need a slightly different version result to one which appears in \cite[Proposition~4.5]{CZ}, we include its proof.
\begin{lemma}\label{l:l09}
Suppose $T_i\colon\Hilbert\to\Hilbert$ is $\rho_i$-SQNE for all $i\in\{1,\dots,n\}$ and denote $T:=\sum_{i=1}^n\omega_iT_i$ where $\sum_{i=1}^n\omega_i=1$ and $\omega_i>0$ for all $i\in\{1,\dots,n\}$. 
Assume that $\Fix T=\cap_{i=1}^n\Fix T_i\neq\emptyset$. Then
 \begin{equation*}
 \sum_{i=1}^n\omega_i\rho_i\|x-T_i(x)\|^2 \leq 2\dist(x,\Fix T)\|x-T(x)\|\quad\forall x\in\Hilbert.
 \end{equation*}
\end{lemma}
\begin{proof}
Let $z=P_{\Fix T}(x)$. Since $T_i$ is $\rho_i$-SQNE, we have
\begin{align*}
  \|T(x)-z\|^2 \leq \sum_{i=1}^n\omega_i\|T_i(x)-z\|^2 \leq \|x-z\|^2 - \sum_{i=1}^n\omega_i\rho_i\|x-T_i(x)\|^2.
\end{align*}
Using the Cauchy--Schwarz inequality, we deduce
\begin{align*}
\|T(x)-z\|^2 
  &= \|T(x)-x\|^2 + \|x-z\|^2 + 2\langle T(x)-x,x-z\rangle \\
  &\geq \|T(x)-x\|^2 + \|x-z\|^2 - 2\dist(x,\Fix T)\|T(x)-x\|.
\end{align*}
The claimed result follows by combining the previous two inequalities.
\end{proof}

\begin{theorem}\label{th:hoelder reg comb}
Let $T_i\colon\Hilbert\to\Hilbert$ be $\rho_i$-SQNE for all $i\in\{1,\dots,n\}$ and assume $\Fix T=\cap_{i=1}^n\Fix T_i\neq\emptyset$. Suppose that the following assertions hold.
\begin{enumerate}[(i)]
\item The operator $T_i$ is boundedly H\"older regular.
\item The collection $\{\Fix T_i\}_{i=1}^n$ is boundedly H\"older regular.
\end{enumerate}
Then $T:=\sum_{i=1}^n\omega_iT_i$ is boundedly H\"older regular whenever $\sum_{i=1}^n\omega_i=1$ and $\omega_i>0$ for all $i\in\{1,\dots,n\}$.
\end{theorem}
\begin{proof}
Let $U$ be a nonempty bounded set. Since $T_i$ is boundedly H\"older regular, there exists constants $\kappa_i>0$ and $\gamma_i\in(0,1)$ such that
\begin{equation*}
  \dist(x,\Fix T_i) \leq \kappa_i\|x-T_i(x)\|^{\gamma_i}\quad\forall x\in U. 
\end{equation*}
Denote $\gamma=\min_{i=1,\dots,n}\gamma_i\in(0,1)$. Since $U$ is bounded and $\gamma\leq\gamma_i$, Lemma~\ref{l:hoelder trick} implies the existence of constants $M_i>0$ such that
$$ \|x-T_i(x)\|^{\gamma_i}\leq M_i\|x-T_i(x)\|^{\gamma}\quad\forall x\in U. $$
Denote $\kappa=\max_{i=1,\dots,n}\kappa_iM_i$. Then combining the previous two inequalities gives
\begin{equation*}
   \dist(x,\Fix T_i) \leq \kappa_iM_i\|x-T_i(x)\|^{\gamma}\leq \kappa\|x-T_i(x)\|^{\gamma} \quad\forall x\in U.
\end{equation*}
Let $x\in U$ and $z\in\Fix T$. Set $\omega=\min_{j=1,\dots,n}\omega_j$ and set $\rho=\min_{j=1,\dots,n}\rho_j$. Then
 $$ \omega_i\dist(x,\Fix T_i) \leq \sum_{j=1}^n\omega_j\dist(x,\Fix T_j) \leq \kappa\sum_{j=1}^n\omega_j\|x-T_jx\|^\gamma, $$
and so convexity of $t\mapsto t^{2/\gamma}$ together with Lemma~\ref{l:l09} implies
\begin{equation}\label{eq:key hoelder 0}
 \omega^{2/\gamma}\dist^{2/\gamma}(x,\Fix T_i) \leq \kappa^{2/\gamma}\sum_{j=1}^n\omega_j\|x-T_jx\|^2 \leq \frac{2\kappa^{2/\gamma}}{\rho}\dist(x,\Fix T)\|x-T(x)\|. 
\end{equation}
Thus, using the fact that the collection $\{\Fix T_i\}_{i=1}^n$ is H\"older regular on $U$ together with \eqref{eq:key hoelder 0}, we deduce the existence of a $\tau>0$ and a $\theta\in(0,1)$ such that
  $$ \dist^{\frac{2}{\gamma\theta}}(x,\Fix T) \leq \tau^{\frac{2}{\gamma\theta}}\max_{i=1,\dots,n}\dist^{2/\gamma}(x,\Fix T_i) \leq \tau^{\frac{2}{\gamma\theta}}\frac{2\kappa^{2/\gamma}}{\rho \omega^{2/\gamma}}\dist(x,\Fix T)\|x-T(x)\|, $$
from which the result follows.    
\end{proof}

The following lemma is due to Cegielski \& Zalas \cite{CZ}. Since we need a slightly different version result to one which appears in \cite[Proposition~4.6]{CZ}, we include its proof.
\begin{lemma}\label{l:l10}
Suppose $T_i\colon\Hilbert\to\Hilbert$ is $\rho_i$-SQNE for all $i\in\{1,\dots,n\}$ and denote $T:=T_n\dots T_2T_1$. Assume that $\Fix T=\cap_{i=1}^n\Fix T_i\neq\emptyset$. Then
 \begin{equation*}
 \sum_{i=1}^n\rho_i\|Q_{i-1}(x)-Q_i(x)\|^2 \leq 2\dist(x,\Fix T)\|x-T(x)\|\quad\forall x\in\Hilbert,
 \end{equation*}
where we denote $Q_0:=\Id$ and $Q_i:=T_i\dots T_1$ for all $i\in\{1,\dots,n\}$.
\end{lemma}
\begin{proof}
Let $z=P_{\Fix T}(x)$. Since $T_i$ is $\rho_i$-SQNE, we have
$$ \|T(x)-z\|^2 \leq \|x-z\|^2 - \sum_{i=1}^n\rho_i\|Q_{i-1}(x)-Q_i(x)\|^2. $$
Using the Cauchy--Schwarz inequality, we have
\begin{align*}
\|T(x)-z\|^2 
  &= \|T(x)-x\|^2 + \|x-z\|^2 + 2\langle T(x)-x,x-z\rangle \\
  &\geq \|T(x)-x\|^2 + \|x-z\|^2 - 2\dist(x,\Fix T)\|T(x)-x\|.
\end{align*}
The claimed result follows by combining the previous two inequalities.
\end{proof}

\begin{theorem}\label{th:hoelder reg comp}
Let $T_i\colon\Hilbert\to\Hilbert$ be $\rho_i$-SQNE for $i\in\{1,\dots,n\}$ and let $T:=T_n\dots T_2 T_1$. Assume that $\Fix T=\cap_{i=1}^n\Fix T_i\neq\emptyset$ and denote $U=\mathbb{B}(z,R)$ for some $z\in\Fix T$ and $R>0$. Suppose that the following assertions hold.
\begin{enumerate}[(i)]
\item The operator $T_i$ is H\"older regular on $U$ for all $i\in\{1,\dots,n\}$.
\item The collection $\{\Fix T_i\}_{i=1}^n$ is H\"older regular on $U$.
\end{enumerate}
Then $T$ is H\"older regular on $U$.
\end{theorem}
\begin{proof}
Denote $Q_0=\Id$ and $Q_i=T_i\dots T_2T_1$ for all $i\in\{1,\dots,n\}$.
Since $T_i$ is H\"older regular on $U$, there exists constants $\kappa_i>0$ and $\gamma_i\in(0,1)$ such that
\begin{equation*}
  \dist(x,\Fix T_i) \leq \kappa_i\|x-T_i(x)\|^{\gamma_i}\quad\forall x\in U. 
\end{equation*}
Denote $\gamma=\min_{i=1,\dots,n}\gamma_i\in(0,1)$. By using the same argument as in Theorem~\ref{th:hoelder reg comb}, we deduce the existence of $\kappa>0$ such that, for all $i\in\{1,\dots,n\}$, we have 
\begin{equation}\label{eq:Ti hoelder}
   \dist(x,\Fix T_i) \leq \kappa\|x-T_i(x)\|^{\gamma}\quad\forall x\in U. 
\end{equation}
Let $x\in U$. Then, since $T_i$ is $\rho_i$-SQNE for all $i\in\{1,\dots,n\}$, we have that 
\begin{align*}
   R^2 \geq \|x-z\|^2 
   &\geq \|Q_1(x)-z\|^2 + \rho_1\|Q_0(x)-Q_1(x)\|^2 \\
   &\geq \|Q_2(x)-z\|^2 + \rho_2\|Q_1(x)-Q_2(x)\|^2 + \rho_1\|Q_0(x)-Q_1(x)\|^2 \\
   &\;\,\vdots \\
   &\geq \|T(x)-z\|^2 + \sum_{i=1}^n\rho_i\|Q_{i-1}(x)-Q_i(x)\|^2.
\end{align*}   
From this, it follows that $Q_i(x)\in U$ for all $i\in\{1,\dots,n\}$ and that 
 $$\max_{i=1,\dots,n}\|Q_{i-1}(x)-Q_i(x)\|\in\left[0,\frac{R}{\sqrt{\rho}}\right]\text{~~where~~}\rho=\min_{i=1,\dots,n}\rho_i. $$
By Lemma~\ref{l:hoelder trick}, there exists a constant $\mu>0$ such that, for all $i\in\{1,\dots,n\}$, we have
\begin{equation}\label{eq:hoelder asym}
  \|Q_{i-1}(x)-Q_i(x)\| \leq  \mu\|Q_{i-1}(x)-Q_i(x)\|^\gamma\text{~~for all~~}x\in U. 
\end{equation}
Set $M:=\max\{\mu,\kappa\}$ and $j\in\{1,\dots,n\}$. Applying the triangle inequality, followed by \eqref{eq:Ti hoelder} and \eqref{eq:hoelder asym} gives
\begin{align*}
\dist(x,\Fix T_j) 
 &\leq \|x-P_{\Fix T_j}(Q_{j-1}(x))\| \\
 &\leq \|x-Q_1(x)\| + \|Q_1(x)-Q_2(x)\| + \dots + \|Q_{j-1}(x)-P_{\Fix T_j}(Q_{j-1}(x))\| \\
 &\leq \|x-Q_1(x)\| + \|Q_1(x)-Q_2(x)\| + \dots + \kappa\|Q_{j-1}(x)-Q_j(x)\|^{\gamma_j} \\
 &\leq M\sum_{i=1}^n\|Q_{i-1}(x)-Q_i(x)\|^{\gamma}.
\end{align*}
Set $\rho:=\min_{i=1,\dots,n}\rho_i$. Using convexity of $t\mapsto t^{2/\gamma}$ followed by Lemma~\ref{l:l10}, we  deduce
\begin{equation}\label{eq:key hoelder}
 \dist^{2/\gamma}(x,\Fix T_j) \leq n^{(\gamma/2-1)}M^{2/\gamma}\sum_{i=1}^n\|Q_{i-1}(x)-Q_i(x)\|^2 \leq \frac{2n^{(\gamma/2-1)}M^{2/\gamma}}{\rho}\dist(x,\Fix T)\|x-T(x)\|. 
\end{equation}
Thus, using the fact that the collection $\{\Fix T_j\}_{j=1}^n$ is H\"older regular on $U$ together with \eqref{eq:key hoelder}, we deduce the existence of a $\tau>0$ and a $\theta\in(0,1)$ such that
  $$ \dist^{\frac{2}{\gamma\theta}}(x,\Fix T) \leq \tau^{\frac{2}{\gamma\theta}}\max_{j=1,\dots,n}\dist^{2/\gamma}(x,\Fix T_j) \leq \frac{2n^{(\gamma/2-1)}M^{2/\gamma}\tau^{\frac{2}{\gamma\theta}}}{\rho}\dist(x,\Fix T)\|x-T(x)\|. $$
The result then follows on observing that $\frac{\gamma\theta}{2-\gamma\theta}<1$ as $\gamma,\theta\in(0,1)$.
\end{proof}

\begin{corollary}\label{cor:hoelder reg comp}
Let $T_i\colon\Hilbert\to\Hilbert$ be $\rho_i$-SQNE for all $i\in\{1,\dots,n\}$. Assume $\cap_{i=1}^n\Fix T_i\neq\emptyset$. Suppose the following assertions hold.
\begin{enumerate}[(i)]
\item The operator $T_i$ is boundedly H\"older regular for all $i\in\{1,\dots,n\}$.
\item The collection $\{\Fix T_i\}_{i=1}^n$ is boundedly H\"older regular.
\end{enumerate}
Then $T:=T_n\dots T_2T_1$ is boundedly H\"older regular.
\end{corollary}
\begin{proof}
Follows by combining Theorem~\ref{th:hoelder reg comp} with Remark~\ref{re:bounded reg equiv}.
\end{proof}

\section{Convergence Rates for Combinations and Compositions}\label{s:convergence2}
In this section, we further refine the results from Section~\ref{s:convergence}. More precisely, we consider the dynamical system \eqref{eq:system} in the setting when the operator $T$ can be expressed in terms of a convex combination or a composition of operators $T_i\colon\Hilbert\to\Hilbert$ for $i\in\{1,\dots,n\}$ with $\cap_{i=1}^n\Fix T_i\neq\emptyset$. In other words, we consider the system
  \begin{equation}\label{eq:system2}
   \dot{x}(t) = \lambda(t)\bigl(T\bigl(x(t)\bigr)-x(t)\bigr), 
  \end{equation}
where $T$ is given by ether:  
\begin{enumerate}[(i)]
\item $T=\sum_{i=1}^n\omega_iT_i$ with $\sum_{i=1}\omega_i=1$ and $\omega_i>0$ for all $i\in\{1,\dots,n\}$, or
\item $T=T_n\dots T_2T_1$. 
\end{enumerate}
Situations of this kind naturally arise in the study of continuous-time \emph{projection algorithms} for solving the \emph{feasibility problem}. This problem asks for a point in the intersection of closed, convex constraints $C_1,\dots,C_n$. In the simplest such algorithm, the \emph{method of cyclic projections}, $T_i=P_{C_i}$ where $P_{C}$ denotes the \emph{nearest point projector} onto a set $C$ given by
  $$ P_{C}(x) = \{c\in C:\|x-c\|\leq\|x-z\|\,\forall z\in C\}, $$
and $T=P_{C_n}\dots P_{C_2}P_{C_1}$ is the \emph{cyclic projections} operator. Another example is provided by \emph{Douglas--Rachford} methods in which each operators $T_i$ is  a \emph{Douglas--Rachford} operator of the form
  $$ \frac{\Id + (2P_{C_j}-\Id)(2P_{C_l}-\Id)}{2} =  \Id + P_{C_j}(2P_{C_l}-\Id) - P_{C_l} $$
for a pair indices $j,l\in\{1,\dots,n\}$. For further details on projection algorithms (with $\Hilbert$ potentially infinite dimensional) in linearly regular settings, see \cite{BNP2015,CRZ2018}, and in H\"older regular settings, see \cite{BLT2017}.

We obtain the results in this section by combining the results from the previous two sections. To do so, we require the following class of operators which are both nonexpansive and strongly quasinonexpansive.

\begin{definition}[Averaged nonexpansive {\cite{BBR}}]
An operator $T\colon\Hilbert\to\Hilbert$ is \emph{$\alpha$-averaged nonexpansive} if $\alpha\in(0,1)$ such that one of the following two equivalent properties holds.
\begin{enumerate}[(i)]
\item There exists a nonexpansive operator $R\colon\Hilbert\to\Hilbert$ such that
  $ T = (1-\alpha)\Id+\alpha R. $
\item For all $x,y\in\Hilbert$, we have
  $$ \|T(x)-T(y)\|^2 + \frac{1-\alpha}{\alpha}\|(\Id-T)(x)-(\Id-T)(y)\|^2 \leq \|x-y\|^2. $$
\end{enumerate} 
\end{definition}

Note that it is immediate from the respective definitions that an $\alpha$-averaged operator is $\rho$-SQNE with $\rho=(1-\alpha)/\alpha$.

\begin{corollary}
Let $T_i\colon\Hilbert\to\Hilbert$ be $\alpha_i$-averaged nonexpansive with $\cap_{i=1}^n\Fix T_i\neq\emptyset$. Suppose $\lambda\colon[0,+\infty)\to[0,1]$ is Lebesgue measurable with $\inf_{t\geq 0}\lambda(t)\bigl(1-\lambda(t)\bigr)>0$. Let $x$ be the unique strong global solution of \eqref{eq:system2}. Further, suppose that the following assertions hold.
\begin{enumerate}[(i)]
\item The operator $T_i$ is boundedly linearly regular for $i\in\{1,\dots,n\}$.
\item The collection $\{\Fix T_i\}_{i=1}^n$ is boundedly linearly regular.
\end{enumerate}
Then there exists $\bar{x}\in\cap_{i=1}^n\Fix T_i$ and constants $M,r>0$ such that, for almost all $t\in[0,+\infty)$, we have
$$ \|x(t)-\bar{x}\| \leq M\exp(-rt). $$
In particular, the trajectory $x(t)$ converges strongly to $\bar{x}$ as $t\to+\infty$. 
\end{corollary}
\begin{proof}
By either Theorem~\ref{th:linear reg comb} or Corollary~\ref{cor:linear reg comp}, the operator $T$ is boundedly H\"older regular. The result then follows by Theorem~\ref{th:linear reg convergence}.
\end{proof}

\begin{corollary}
Let $T_i\colon\Hilbert\to\Hilbert$ be $\alpha_i$-averaged nonexpansive with $\cap_{i=1}^n\Fix T_i\neq\emptyset$. Suppose $\lambda\colon[0,+\infty)\to[0,1]$ is Lebesgue measurable with $\inf_{t\geq 0}\lambda(t)\bigl(1-\lambda(t)\bigr)>0$. Let $x$ be the unique strong global solution of \eqref{eq:system2}. Further, suppose that the following assertions hold.
\begin{enumerate}[(i)]
\item The operator $T_i$ is boundedly H\"older regular for all $i\in\{1,\dots,n\}$.
\item The collection $\{\Fix T_i\}_{i=1}^n$ is boundedly H\"older regular.
\end{enumerate}
Then there exists $\bar{x}\in\cap_{i=1}^n\Fix T_i$, $M>0$ and $\gamma\in(0,1)$ such that, for almost all $t\in[0,+\infty)$, we have
$$ \|x(t)-\bar{x}\| \leq M\,t^{-\frac{\gamma}{2(1-\gamma)}}. $$
In particular, the trajectory $x(t)$ converges strongly to $\bar{x}$ as $t\to+\infty$. 
\end{corollary}
\begin{proof}
By either Theorem~\ref{th:hoelder reg comb} or Corollary~\ref{cor:hoelder reg comp}, the operator $T$ is boundedly H\"older regular. The result then follows by Theorem~\ref{th:hoelder reg convergence}.
\end{proof}

\paragraph*{Acknowledgements}
The first author is supported by FWF (Austrian Science Fund) project P 29809-N32.
The second author is supported in part by ARC grant DP200101197.
The third author is supported in part by ARC grant DE200100063.
The authors would like to thank the anonymous referees for their helpful comments, which included an improvement to Theorem~\ref{th:linear reg convergence}.

\end{document}